\documentclass{amsart}[12 pt]

\usepackage{latexsym}

\usepackage{amsmath,amssymb, amsfonts}
\usepackage{verbatim}
\usepackage{graphicx}
\usepackage{graphics}
\usepackage{geometry}
\usepackage{booktabs}

\newtheorem{proposition}{Proposition}[section]
\newtheorem{theorem}{Theorem}[section]
\newtheorem{definition}{Definition}[section]
\newtheorem{corollary}{Corollary}[section]
\newtheorem{lemma}{Lemma}[section]
\newtheorem{remark}{Remark}[section]

\numberwithin{equation}{section}

\allowdisplaybreaks

\begin{document}
\markboth{}{}

\title[A Maximal Inequality of the 2D Young Integral based on Bivariations] {A Maximal Inequality of the 2D Young Integral based on Bivariations}



\author{Alberto Ohashi}

\address{Departamento de Matem\'atica, Universidade Federal da Para\'iba, 13560-970, Jo\~ao Pessoa - Para\'iba, Brazil}\email{alberto.ohashi@pq.cnpq.br; ohashi@mat.ufpb.br}

\author{Alexandre B. Simas}

\address{Departamento de Matem\'atica, Universidade Federal da Para\'iba, 13560-970, Jo\~ao Pessoa - Para\'iba, Brazil}\email{alexandre@mat.ufpb.br}

\date{\today}

 \subjclass{}

\begin{center}

\end{center}

\begin{abstract}
In this note, we establish a novel maximal inequality of the 2D Young integral $\int_a^b\int_c^d FdG$ in terms of the $(p,q)$-bivariation norms of the
section functions $x\mapsto F(x,y)$ and $y\mapsto F(x,y)$ where $G:[a,b]\times [c,d]\rightarrow \mathbb{R}$ is a controlled path satisfying finite $(p,q)$-variation conditions. The proof is reminiscent from the Young's original ideas~\cite{young1} in defining two-parameter integrals in terms of
$(p,q)$-finite bivariations. Our result complements the standard maximal inequality established by Towghi~\cite{towghi1} in terms of joint variations. We apply the maximal inequality to get novel strong approximations for 2D Young integrals w.r.t the Brownian local time in terms of number of upcrossings of a given approximating random walk.
\end{abstract}

\maketitle

\section{Preliminaries and Main Result}
One remarkable result in the seminal L.C Young's article~\cite{young} is the development of a (1-parameter) 1D Riemman-Stieltjes-type integral $\int fdg$ where $f,g:[a,b]\rightarrow \mathbb{R}$ are two functions with suitable finite variations (see e.g~\cite{friz})

$$\|f\|^p_{[a,b],p}:= \sup_{\Pi}\sum_{x_i\in \Pi}|f(x_i)-f(x_{i-1})|^p <\infty, \quad\hbox{and}\quad \|g\|^q_{[a,b],q}:= \sup_{\Pi}\sum_{x_i\in \Pi}|g(x_i)-g(x_{i-1})|^q <\infty,$$
where $\frac{1}{p} + \frac{1}{q}> 1$ and $\sup$ is taken over all partitions of the compact set $[a,b]\subset \mathbb{R}$. Let $\mathcal{W}^p([a,b];\mathbb{R})$ be the linear space of real-valued functions $h$ equipped with the seminorm $\|h\|_{[a,b],p} < \infty$. In his seminal article in 1936, Young proved that if $f\in \mathcal{W}^p([a,b];\mathbb{R})$ and $g\in \mathcal{W}^q([a,b];\mathbb{R})$ are two continuous functions, then there exists an absolute constant $C>0$ such that

\begin{equation}\label{in1}
\Big|\int_a^bfdg \Big|\le C \Big[|f(a)| + \|f\|_{[a,b],p}\Big]\|g\|_{[a,b],q},
\end{equation}
provided $\frac{1}{p} + \frac{1}{q}> 1$.

The 1D Young's integration theory has great importance in many areas in Analysis and Probability. In particular, it was the starting point for T. Lyons~(see e.g~\cite{lyons}) to introduce his original ideas on the so-called Rough Path theory where higher-order increments of the functions play a key role in determining integrals beyond the constraint $\frac{1}{p} + \frac{1}{q} > 1$. The 2D Young integral was introduced by L.C Young~\cite{young1} and recently it has been an important tool in the Gaussian rough path theory~\cite{friz2,hairer,gubinelli}, extensions of It\^o formula~\cite{feng} and functional stochastic calculus~\cite{LOS}. See also~\cite{towghi1} for a particular multi-dimensional extension of the 2D Young integral. In the sequel, let us recall some basic definitions from the original article~\cite{young1}.

Throughout this article, we are going to fix $-\infty < a < b < +\infty$ and $-\infty < c < d < +\infty$. Let $\Pi(\xi)=\{x_i; 0\le i \le N\}$ be a partition of $[a,b]$ equipped with a set of points $\xi=\{\xi_i; i=1,\ldots,N\}$, where $x_{i-1}\le \xi_i\le x_i$ and $x_0=a,~x_N=b$. We call $\Pi(\xi)$ a tagged partition of $[a,b]$. Similarly, let $\Pi'(\eta)=\{y_j; 0\le j\le N'\}$ be a partition of $[c,d]$ equipped with a set of points $\eta=\{\eta_j; j=1,\ldots,N'\}$ such that $y_{j-1}\le \eta_j\le y_j$ and $y_0=c$ and $y_{N'}=d$. We call $\Pi'(\eta)$ a tagged partition of $[c,d]$.

\begin{definition}
Let $\Pi(\xi)$ be a tagged partition of $[a,b]$ and let
$f:[a,b]\to\mathbb{R}$ be a given function. We say $f_{\Pi(\xi)}:[a,b]\to\mathbb{R}$ is a \textbf{step function} of $f$ based on $\Pi(\xi)$ if $f_{\Pi(\xi)}(x_i) = f(x_i)$, for $i=0,\ldots, N$, and $f_{\Pi(\xi)}(x) = f(\xi_i)$, if $x_{i-1}<x<x_i$, $i=1,\ldots,N$.
\end{definition}

It is immediate from the definition that $\|f_{\Pi(\xi)}\|_{[a,b],p}\le \|f\|_{[a,b],p}$ for every tagged partition $\Pi(\xi)$. Throughout this article, we make use of the following terminology. $\Pi(\xi)$ and $\Pi'(\eta)$ will denote tagged partitions, whereas the notations $\Pi_j$ and $\Pi_j'$ stand for ``untagged'' partitions of $[a,b]$ and $[c,d]$, respectively. Finally, we say that an ``untagged'' partition $\Pi$ refines a tagged partition $\Pi(\xi)$, if $\Pi$ refines the partition $\Pi(\xi)$ without taking into account the set $\xi$.

\begin{definition}
Let $\Pi(\xi)$ and $\Pi'(\eta)$ be tagged partitions of $[a,b]$ and $[c,d]$, respectively, and let $F:[a,b]\times [c,d]\to\mathbb{R}$ be a given function. We say that $F_{\Pi(\xi),\Pi'(\eta)}:[a,b]\times [c,d]\to\mathbb{R}$ is a \textbf{step function of} $F$ \textbf{on} $(\Pi(\xi),\Pi'(\eta))$, if $F_{\Pi(\xi),\Pi'(\eta)}(x_i,y_j) = F(x_i,y_j),$ for $i=0,\ldots,N$, $j=0,\ldots,N'$; $F_{\Pi(\xi),\Pi'(\eta)}(x_i,y) = F(x_i,\eta_j)$, if $y_{j-1}<y<y_j$, $i=0,\ldots,N$ and $j=1,\ldots,N'$; $F_{\Pi(\xi),\Pi'(\eta)}(x,y_j) = F(\xi_i,y_j),$ if $x_{i-1}<x<x_i$, $i=1,\ldots,N$ and $j=0,\ldots,N'$; $F_{\Pi(\xi),\Pi'(\eta)}(x,y) = F(\xi_i,\eta_j)$ if $x_{i-1} < x < x_i$ and $y_{j-1}< y < y_j$ for $i=1,\ldots,N$ and $j=1,\ldots, N'$.
\end{definition}

In the sequel, $\Delta_iH(x_i,y_j): = H(x_i,y_j)-H(x_{i-1},y_j)$ denotes the first difference operator acting on the variable $x$ of a
given function $H:[a,b]\times[c,d]\rightarrow \mathbb{R}$, whereas $\Delta_jH(x_i,y_j) = H(x_i,y_j)-H(x_i,y_{j-1})$ denotes the first
difference operator acting on the variable $y$ of $H$.

Let $F_{\Pi(\xi),\Pi'(\eta)}$ be a step function for a given $F$ and let $G$ be a function such that $(x,y)\mapsto G(x,y)-G(\alpha,y) -G(x,\beta)+G(\alpha,\beta) $ admits only points of discontinuity of first kind for any $\alpha\in [a,b],~\beta\in [c,d]$. Then, we define

$$
\int_a^b\int_c^d F_{\Pi(\xi),\Pi'(\eta)}(x,y)d_{x,y}G(x,y):=\sum_{j=1}^N\sum_{i=1}^{N'}F(\xi_{i},\eta_{j})\Delta_i\Delta_j G(x_i,y_j).
$$
We are now in position to recall the classical definition of the 2D Young integral, see also Section 4 in~\cite{young1}. Let $F,G:[a,b]\times [c,d]\rightarrow \mathbb{R}$ be two functions (which we would like to emphasize that these functions do not need to be continuous).
We say that the Young integral
$$\int_a^b\int_c^d F(x,y)d_{x,y}G(x,y)$$
exists (in the generalized Moore-Pollard sense) and it is equal to a real number $I$ if for every $\epsilon>0$, there exist finite subsets $E$ and $E'$ of $[a,b]$ and $[c,d]$, respectively, such that
for every tagged partitions $\Pi(\xi)$ and $\Pi'(\eta)$, with the partition $\Pi(\xi)$ containing the points in $E$, and $\Pi'(\eta)$ containing
the points in $E'$, the Riemann-Stieltjes integral of the step function $F_{\Pi(\xi),\Pi'(\eta)}$ with respect to $G$ satisfies
$$\left| \int_a^b\int_c^d F_{\Pi(\xi),\Pi'(\eta)}(x,y)d_{x,y}G(x,y) - I\right| < \epsilon.$$


The following notion is originally due to Young~\cite{young1} and it will play a key role in this work:
\begin{definition}\label{pq}
We say that $F:[a,b]\times [c,d]\rightarrow \mathbb{R}$ has $(p,q)$-bivariation for $p,q >0$ if

$$\|F\|_{1;p}:=\sup_{y_1,y_2\in [c,d]^2}\| F(\cdot, y_1) - F(\cdot, y_2)\|_{[a,b],p}< \infty,$$
and
$$\|F\|_{2;q}:=\sup_{x_1,x_2\in [a,b]^2}\| F(x_1, \cdot) - F(x_2,\cdot)\|_{[c,d],q}< \infty.$$
\end{definition}



There is a very related notion of variation which takes into account joint variation in both variables rather than bivariations  ~(see~e.g~\cite{friz,friz1,towghi,towghi1}):

\begin{definition}\label{jointdef}
Let $p\in [ 1, \infty)$. A  function $F:[a,b]\times [c,d]\rightarrow \mathbb{R}$ has finite $p$-variation if

$$V_p(F): = \Bigg(\sup_{\Pi,\Pi'}\sum_{\substack{x_i\in \Pi\\y_j\in\Pi'}} |\Delta_i\Delta_j F|^p\Bigg)^{\frac{1}{p}}< \infty,$$
where the supremum varies over all partitions $\Pi$ of $[a,b]$ and $\Pi'$ of $[c,d]$.
\end{definition}
The linear space of real-valued functions defined on $[a,b]\times [c,d]$ having finite $p$-variation equipped with the seminorm $V_p(F)$ will be denoted by $\mathcal{W}^p([a,b]\times [c,d];\mathbb{R})$. The following remarks show that the the joint variation notion is actually stronger than bivariations.

\begin{proposition}
 Let $F:[a,b]\times [c,d]\to \mathbb{R}$, then
 $$\|F\|_{1;p}\leq V_p(F)\quad \text{and}\quad \|F\|_{2;q} \leq V_q(F).$$
\end{proposition}
\begin{proof}
We will only prove the first inequality, since the other one is entirely analogous. In the case where the supremum is attained at $y_1=c$ and $y_2=d$, the inequality is obvious. Let us assume $c\le y_1< y_2\le d$, and consider the partition $\Pi'= \{s_j\}$,
with $s_0=c; s_1=y_1; s_2 =y_2; s_3 = d$. Then, we have
\begin{eqnarray*}
\|F(\cdot,y_1) - F(\cdot,y_2)\|_{[a,b],p}^p &=& \sup_{\Pi} \sum_{x_i\in\Pi} |F(x_i,y_2)-F(x_i,y_1)-F(x_{i-1},y_2) + F(x_{i-1},y_1)|^p\\
&=& \sup_{\Pi}\sum_{x_i\in\Pi} |F(x_i,s_2)-F(x_i,s_1)-F(x_{i-1},s_2) + F(x_{i-1},s_1)|^p\\
&\leq& \sup_{\Pi}\sum_{s_j\in \Pi'} \sum_{x_i\in\Pi} |\Delta_i\Delta_j F(x_i,s_j)|^p\\
&\leq&V^p_p(F).
\end{eqnarray*}
Now, one can take the supremum on the left-hand side of the inequality. This concludes the proof.
\end{proof}

\begin{remark}
One should notice that $F\in \mathcal{W}^p([a,b]\times [c,d];\mathbb{R})$ if, and only if, both section functions $x\mapsto F(x,\cdot)$ and $y\mapsto F(\cdot,y)$ are $\mathcal{W}^p([\gamma,\eta];\mathbb{R})$-valued $p$-variation functions where $\gamma=c,\eta=d$ and $\gamma=a,\eta=b$, respectively. Moreover, if $F(a,\cdot)=F(\cdot,c)=0$ vanish, then $F$ has $(p,q)$-bivariation if, and only if, $x\mapsto F(x,\cdot)$ and
$y\mapsto F(\cdot,y)$ are $\mathcal{W}^q$-valued bounded resp. $\mathcal{W}^p$-valued bounded functions. In this case, both $\big(\sup_{x\in [a,b]}\|F(x,\cdot)\|_{[c,d],q},\|F\|_{2;q}\big)$ and $\big(\sup_{y\in [c,d]}\|F(\cdot,y)\|_{[a,b],p},\|F\|_{1;p}\big)$ are equivalent.
\end{remark}

The importance of $(p,q)$-bivariation lies in the following result, which is a particular case of a theorem due to L. C. Young.

\begin{theorem}[Young~\cite{young1}, Th. 6.3]\label{tyoung}
Let $p,q>0$, and let $\rho$, $\sigma$, $\mu$ and $\lambda$ be monotone increasing functions such that, $\rho$ and $\sigma$ are subject
to $\rho(u)\sigma(u)=u$. Assume that

\begin{equation}\label{i}
\sum_{n=1}^\infty \rho\left(\frac{1}{n^{\frac{1}{p}}}\right)\lambda\left(\frac{1}{n}\right)<\infty~~\hbox{and}~~\sum_{n=1}^\infty \sigma\left(\frac{1}{n^{\frac{1}{q}}}\right)\mu\left(\frac{1}{n}\right)<\infty.
\end{equation}

\noindent Let $F:[a,b]\times [c,d]\rightarrow \mathbb{R}$ be a function which vanishes on the lines $x=a$ and $y=c$ and which has bounded $(p,q)$-bivariation. Let $G
:[a,b]\times [c,d]\rightarrow \mathbb{R}$ be a function satisfying

\begin{equation}\label{s}
|\Delta_i\Delta_j G(x_i,y_j)|\le \lambda(x_i-x_{i-1}) \mu(y_j-y_{j-1}).
\end{equation}
Then the 2D Young integral $\int_a^b\int_c^d FdG$ exists. That is, for each $\epsilon>0$, there exist finite subsets $E\subset [a,b]$ and $E^{'}\subset [c,d]$ such that
$$\left|\int_a^b\int_c^d F(x,y)d_{x,y}G(x,y) -   \int_a^b\int_c^d F_{\Pi(\xi),\Pi'(\eta)}(x,y)d_{x,y}G(x,y)\right| < \epsilon$$
for every tagged partitions $\Pi(\xi)$ and $\Pi'(\eta)$ which contain points of $E$ and $E'$, respectively.
\end{theorem}

\begin{remark}\label{tychoices}
Typical candidates for the monotone increasing functions above are $\lambda(u) = u^{\frac{1}{\tilde{p}}
},~\mu(u) = u^{\frac{1}{\tilde{q}}}$, $\rho(u)=u^{\alpha}$, $\sigma(u)=u^{1-\alpha}$ in such way that~(\ref{i}) and~(\ref{s}) hold. In the modern language of rough path theory, assumption~(\ref{s}) precisely says that if $\tilde{p}=\tilde{q}$ then $G$ admits a 2D-control $\omega([x_1,x_2]\times [y_1,y_2]) = |x_1-x_2|^{\frac{1}{\tilde{p}}}|y_1-y_2|^{\frac{1}{\tilde{q}}}$ so that~(\ref{s}) trivially implies that $G\in \mathcal{W}^{\tilde{p}}([a,b]\times [c,d];\mathbb{R})$. See Section 5.5 in~\cite{friz}.
\end{remark}


The following result due to Towghi~\cite{towghi} yields the existence of the Young integral under joint variation assumptions for both integrand and integrator as follows. Next, for the convenience of the reader we present his result as stated in~\cite{friz1}.

\begin{theorem}[Towghi \cite{towghi}]\label{tow}
Let $p,q \ge 1$, assume that $\theta = \frac{1}{p}+\frac{1}{q}> 1$, and consider $F,G:[a,b]\times [c,d]\rightarrow \mathbb{R}$ functions of $p$-variation resp. $q$-variation which do not have common jump points and $F(a,\cdot) = 0,~F(\cdot, c)=0$. Then the 2D Young integral $\int_a^b\int_c^d FdG$ exists (in the Riemann-Stieltjes sense) and for every $\alpha\in (1,\theta)$,

\begin{equation}\label{boundint}
\Bigg| \int_a^b\int_c^d FdG\Bigg|\le \Bigg[\Bigg(1+\zeta \left(\frac{\theta}{\alpha}\right)\Bigg)^\alpha \zeta(\alpha) +(1+\zeta(\theta))\Bigg] V_p(F) V_q(G),
\end{equation}
where $\zeta(s) = \sum_{n=1}^\infty \frac{1}{n^s}.$
\end{theorem}

By comparing Theorems~\ref{tow} and~\ref{tyoung}, we notice that the price we pay when dealing with $(p,q)$-bivariation is the stronger assumption~(\ref{s}) which provides the necessary smoothness on $G$ in order to get the existence of 2D Young integral. In one hand, one should notice that the mere finiteness of $V_q(G)<\infty$ does not imply~(\ref{s}) with $\lambda(u)= \mu(u)= u^{\frac{1}{p}}$. On the other hand, when $G$ satisfies~(\ref{s}) then we shall relax the joint variation property in $F$ by requiring only finite bivariation of a suitable order. Therefore, it is natural to ask a maximal inequality for the 2D Young integral under assumptions in Theorem~\ref{tyoung}. This issue is particularly important in the theory of local-times of Brownian motion. See Section~\ref{applt} for further details.

\subsection{Main Result}
In this note, our goal is to establish a maximal inequality for the 2D Young integral in terms of $(p,q)$-bivariations rather than the joint variation notion of Def~\ref{jointdef} and Theorem~\ref{tow}. We explore the $(p,q)$-bivariation notion pioneered by L.C Young instead of the joint variation in order to obtain the following maximal inequality for the 2D Young integral.



\begin{theorem}\label{mainTh}
Under the assumptions of Theorem~\ref{tyoung}, the following estimate holds

\begin{eqnarray}
\left|\int_{a}^{b}\int_{c}^{d} F(x,y)d_{x,y}G(x,y)-F(b,d)(G(b,d)-G(b,c)-G(a,d)+G(a,c))\right|\\\nonumber
 \leq K \left(\sum_{m=1}^\infty \rho\left(\frac{\|F\|_{1;p}}{m^{1/p}}\right)\lambda\left(\frac{4}{m}\right)\right)\left(\sum_{m'=1}^\infty\sigma\left(\frac{\|F\|_{2;q}}{{m'}^{1/q}}\right)\mu\left(\frac{4}{m'}\right)\right)\\\nonumber
 +K_1\mu(d-c) \sum_{m=1}^\infty\frac{\|F\|_{1;p}}{m^{1/p}}\lambda\left(\frac{4}{m}\right)+ K_2\lambda(b-a) \sum_{m'=1}^\infty\frac{\|F\|_{2;q}}{{m'}^{1/q}}\mu\left(\frac{4}{m'}\right).\label{mainbound}
\end{eqnarray}
where $K,K_1$ and $K_2$ are absolute constants.
\end{theorem}
The first term on the right-hand side of the above inequality can be seen as a mixture of the $(p,q)$-bivariations, whereas the other terms are
purely marginal terms. One observes that the joint $p$-variation is replaced by an equilibrium of the marginal $(p,q)$-bivariations, with the equilibrium
being given by the functions $\rho$ and $\sigma$. This relaxation is compensated by controlling the paths of $G$ by means of assumption~(\ref{s}).


In most applications of Theorem \ref{mainTh}, the statement can be simplified. In fact, as indicated in Remark~\ref{tychoices}, the typical candidates for the functions $\rho(u)$ and $\sigma(u)$ are given by $\rho(u) = u^\alpha$, $\sigma(u) = u^{1-\alpha}$. Furthermore, the functions $\lambda(u)$ and $\mu(u)$ are usually given by $\lambda(u) = u^{1/\tilde{p}}$, and $\mu(u) = u^{1/\tilde{q}}$, with $\tilde{p},\tilde{q}>1$. In this case, we have the following corollary.

\begin{corollary}\label{mainbound}
Let $F,G:[a,b]\times [c,d]\rightarrow \mathbb{R}$ be two functions, where $F$ vanishes on the lines $x=a$ and $y=c$ and has
bounded $(p,q)$-bivariation, and $G$ satisfies $|\Delta_i\Delta_jG(x_i,y_j)|\leq C|x_i-x_{i-1}|^{1/\tilde{p}}|y_j-y_{j-1}|^{1/\tilde{q}}$, for
some constant $C>0$, and $\tilde{p},\tilde{q}>1$. If there exists $\alpha\in (0,1)$ such that
$$\alpha/p + 1/\tilde{p}>1\quad \hbox{and}\quad(1-\alpha)/q + 1/\tilde{q}>1,$$
then, the 2D Young integral $\int_a^b\int_c^d F(x,y)d_{x,y}G(x,y)$ exists and the following estimate
holds
\begin{eqnarray*}
&&\left|\int_{a}^{b}\int_{c}^{d} F(x,y)d_{x,y}G(x,y)-F(b,d)(G(b,d)-G(b,c)-G(a,d)+G(a,c))\right|\\
 &\leq& K(\alpha,p,\tilde{p},q,\tilde{q}) \|F\|_{1;p}^{\alpha}\|F\|_{2;q}^{1-\alpha}+ K_1(p,\tilde{p},\tilde{q})\|F\|_{1;p}+ K_2(q,\tilde{q},\tilde{p})\|F\|_{2;q},
 \end{eqnarray*}
 where
 $$K(\alpha,p,\tilde{p},q,\tilde{q}) = K 4^{1/\tilde{p}}\zeta\left(\frac{\alpha}{p}+\frac{1}{\tilde{p}}\right)\zeta\left(\frac{1-\alpha}{q} +\frac{1}{\tilde{q}}\right),$$
 $$K_1(p,\tilde{p},\tilde{q}) = K_14^{1/\tilde{p}}(d-c)^{1/\tilde{q}}\zeta\left(\frac{1}{p}+\frac{1}{\tilde{p}}\right),$$
$$K_2(q,\tilde{q},\tilde{p}) = K_24^{1/\tilde{q}}(b-a)^{1/\tilde{p}} \zeta\left(\frac{1}{q}+\frac{1}{\tilde{q}}\right),$$
and $\zeta(s) = \sum_{i=1}^\infty \frac{1}{n^s}.$
\end{corollary}

The importance of Theorem~\ref{mainTh} and~Corollary~\ref{mainbound} lies in cases when $F$ lacks or it is hard to check joint variation but $G$ satisfies condition~(\ref{s}). This type of regularity naturally arises in the context of functional It\^o formulas~(see e.g~\cite{LOS}). See Section~\ref{applt} for some examples related to space-time local-time integral in the Brownian motion setting.

\section{Proof of Theorem~\ref{mainTh}}
Throughout this section, we are going to fix a function $G:[a,b]\times[c,d]\rightarrow \mathbb{R}$ such that $(x,y)\mapsto G(x,y)-G(\alpha,y) -G(x,\beta)+G(\alpha,\beta)$ admits only points of discontinuity of first kind for any $\alpha\in [a,b],~\beta\in [c,d]$. Let $[a,b]\subset\mathbb{R}$, and let $\Pi(\xi)$ be a fixed tagged partition of $[a,b]$. Denote the points of the partition $\Pi(\xi)$ by $x_0,\ldots,x_N$. We will now obtain a new sequence of partitions $\Pi_0,\ldots,\Pi_M$, where $M$ is chosen in such a way that $\Pi_M$ refines $\{x_0,\ldots,x_N\}$. Each partition $\Pi_j$ is chosen such that $\#\Pi_j = 2^j +1; j\ge 1$ and they are constructed inductively. Let $\Pi_0 = \{a,b\}$. Suppose $\Pi_j = \{t_0^{(j)},\ldots,t_{2^j}^{(j)}\}$,
and $|t_{k+1}^{(j)} - t_k^{(j)}|<4\cdot2^{-j}$, for $k=0,\ldots, 2^j-1$. Then, let for each $k$,  $t^{(j+1)}_{2k} = t_k^{(j)}$,
and $t_{2k+1}^{(j+1)}$ be any
number in $\{x_0,\ldots,x_N\}\cap (t_{k}^{(j)},t_{k+1}^{(j)})$ such that
\begin{equation}\label{cond1}
|t_{2k+1}^{(j+1)} - t_k^{(j)}|<4\cdot2^{-(j+1)}\qquad\hbox{and}\qquad |t_{k+1}^{(j)}-t_{2k+1}^{(j+1)}|<4\cdot 2^{-(j+1)}.
\end{equation}
If there is no such element we take $t_{2k+1}^{(j+1)}:= \frac{t_{k+1}^{(j)} + t_k^{(j)}}{2}$, which obviously satisfies \eqref{cond1}.

It is clear that, since $\#\Pi_M = 2^M+1$, the mesh $\|\Pi_M\|<2^{-(M-2)}$, and $\Pi(\xi)$ is finite, there exists some $M<\infty$ such that $\Pi_M$
refines $\{x_0,\ldots,x_N\}$.
Then, we clearly have
$$\int_a^b f_{\Pi(\xi)} dg = \sum_{i=0}^{2^M} f_{\Pi(\xi)}(t_i^{(M)}) (g(t_{i}^{(M)})-g(t_{i-1}^{(M)})).$$

\begin{lemma}\label{onedim}
Let $\{t_i^{(n)}\}$ be the points of the partition $\Pi_n$, then, for functions $f,g:[a,b]\to\mathbb{R}$, let
$$S_n = \sum_{i=1}^{2^n} f(t_i^{(n)}) (g(t_{i}^{(n)})-g(t_{i-1}^{(n)})).$$
Thus,
$$S_n - S_{n-1} = \sum_{j=1}^{2^{n-1}} (f(t_{2j-1}^{(n)})-f(t_{2j}^{(n)}))(g(t_{2j-1}^{(n)})- g(t_{2j-2}^{(n)})).$$
\end{lemma}
\begin{proof}
Note that
\begin{eqnarray*}
S_{n-1} &=& \sum_{i=1}^{2^{n-1}} f(t_i^{(n-1)}) (g(t_{i}^{(n-1)})-g(t_{i-1}^{(n-1)}))\\
&=& \sum_{i=1}^{2^{n-1}} f(t_{2i}^{(n)}) (g(t_{2i}^{(n)})-g(t_{2i-2}^{(n)}))\\
&=& \sum_{i=1}^{2^{n-1}} f(t_{2i}^{(n)}) (g(t_{2i}^{(n)})-g(t_{2i-1}^{(n)}))+ \sum_{i=1}^{2^{n-1}} f(t_{2i}^{(n)}) (g(t_{2i-1}^{(n)})-g(t_{2i-2}^{(n)})).
\end{eqnarray*}
Since
\begin{eqnarray*}
S_n &=& \sum_{i=1}^{2^n} f(t_i^{(n)}) (g(t_{i}^{(n)})-g(t_{i-1}^{(n)}))\\
&=& \sum_{i=1}^{2^{n-1}} f(t_{2i}^{(n)}) (g(t_{2i}^{(n)})-g(t_{2i-1}^{(n)}))+ \sum_{i=1}^{2^{n-1}} f(t_{2i-1}^{(n)}) (g(t_{2i-1}^{(n)})-g(t_{2i-2}^{(n)})),
\end{eqnarray*}
we have,
\begin{eqnarray*}
S_n - S_{n-1} &=& \sum_{i=1}^{2^{n-1}} (f(t_{2i-1}^{(n)})-f(t_{2i}^{(n)}))(g(t_{2i-1}^{(n)})-g(t_{2i-2}^{(n)})).
\end{eqnarray*}
\end{proof}

We will now prove a two-parameter version of this lemma. We begin with some definitions. Let $F,G:[a,b]\times [c,d]\to\mathbb{R}$ be two functions, $\Pi(\xi), \Pi'(\eta)$ tagged partitions, together with sequences of partitions $\Pi_0,\ldots,\Pi_M$, $\Pi_0',\ldots,\Pi_{M'}'$, where we denote $\Pi_k = \{t_i^{(k)}\}$ and $\Pi_l'= \{s_j^{(l)}\}$. As before, it is easy to see that
$$\int_a^b \int_c^d F_{\Pi(\xi),\Pi'(\eta)}(x,y)d_{x,y} G(x,y) = \sum_{i=1}^{2^M}\sum_{j=1}^{2^{M'}} F_{\Pi(\xi),\Pi'(\eta)}(t_i^{(M)},s_j^{(M')}) \Delta_i\Delta_j G(t_i^{(M)},s_j^{(M')}),$$
where $M$ and $M'$ are such that $\Pi_M$ and $\Pi_{M'}$ refine $\Pi(\xi)$ and $\Pi'(\eta)$, respectively.

For a two-indexed sequence $S_{n,n'}$, we denote $\Delta_1 S_{n,n'} := S_{n,n'}- S_{n-1,n'}$ and $\Delta_2 S_{n,n'} := S_{n,n'} - S_{n,n'-1}$. Then, we have $S_{n,n'} - S_{n-1,n'}-S_{n,n'-1}+S_{n-1,n'-1} = \Delta_1\Delta_2 S_{n,n'}$.

\begin{lemma}\label{d1d2S}
Let $F,G:[a,b]\times [c,d]\to\mathbb{R}$ be two functions, and $\Pi_0,\ldots,\Pi_M$, $\Pi_0',\ldots,\Pi_{M'}'$ sequences of partitions of $[a,b]$
and $[c,d]$, respectively. Then, if we denote
$$S_{n,n'} = \sum_{i=1}^{2^n}\sum_{j=1}^{2^{n'}} F(t_i^{(n)},s_j^{(n')}) \Delta_i\Delta_j G(t_i^{(n)},s_j^{(n')}),$$
we have that
$$\Delta_1\Delta_2 S_{n,n'} = -\sum_{i=1}^{2^n-1}\sum_{j=1}^{2^{n'}-1} \Delta_i\Delta_jF(t_{2i}^{(n)},s_{2j}^{(n')})\Delta_i\Delta_j G(t_{2i-1}^{(n)},s_{2j-1}^{(n')}).$$
\end{lemma}
\begin{proof}
By setting $f_i(y) = F(t_i^{(n)},y)$ and $g_i(y) = \Delta_i G(t_i^{(n)},y)$  in Lemma \ref{onedim}, it is easy to see that for each $i$,
\begin{eqnarray*}
\sum_{j=1}^{2^{n'}} F(t_i^{(n)},s_j^{(n')}) \Delta_i\Delta_j G(t_i^{(n)},s_j^{(n')}) - \sum_{j=1}^{2^{n'-1}} F(t_i^{(n)},s_j^{(n'-1)}) \Delta_i\Delta_j G(t_i^{(n)},s_j^{(n'-1)})\\
= \sum_{j=1}^{2^{n'-1}} (F(t_i^{(n)},s_{2j-1}^{(n')})-F(t_i^{(n)},s_{2j}^{(n')}))(\Delta_i G(t_i^{(n)},s_{2j-1}^{(n')}) - \Delta_i G(t_i^{(n)},s_{2j-2}^{(n')})).
\end{eqnarray*}
Thus, we have that
\begin{eqnarray*}
\Delta_2 S_{n,n'} &=& S_{n,n'}-S_{n,n'-1}\\
&=& \sum_{i=1}^{2^n}\sum_{j=1}^{2^{n'-1}} (F(t_i^{(n)},s_{2j-1}^{(n')})-F(t_i^{(n)},s_{2j}^{(n')}))(\Delta_i G(t_i^{(n)},s_{2j-1}^{(n')}) - \Delta_i G(t_i^{(n)},s_{2j-2}^{(n')})).
\end{eqnarray*}

Applying Lemma \ref{onedim} again, we can proceed in a similar manner to obtain the desired result:
$$\Delta_1\Delta_2 S_{n,n'} =$$
$$\sum_{i=1}^{2^n-1}\sum_{j=1}^{2^{n'}-1} (F(t_{2i-1}^{(n)},s_{2j-1}^{(n')})-F(t_{2i-1}^{(n)},s_{2j}^{(n')})-F(t_{2i}^{(n)},s_{2j-1}^{(n')})+F(t_{2i}^{(n)},s_{2j}^{(n')}))\Delta_i\Delta_j G(t_{2i-1}^{(n)},s_{2j-1}^{(n')}).$$
\end{proof}

\noindent We recall the following elementary remark for reader's convenience.
\begin{lemma}\label{boundsigma}
Let $A,B,C\geq 0$, and let $\alpha>0$. Let $\rho,\sigma:[0,\infty)\to [0,\infty)$ be two non-decreasing functions such that $\rho(u)\sigma(u) = u$. Then,
$$A\leq \alpha B\quad\hbox{and}\quad A\leq \alpha C \Rightarrow A \leq \alpha \rho(B)\sigma(C).$$
\end{lemma}


\noindent Let us now present a suitable bound for the double difference $\Delta_i\Delta_j F$ in terms of bivariations and a pair of monotone functions.

\begin{lemma}\label{boundd1d2f}
 Let $\Pi = \{t_0,\ldots,t_{2^m}\}$ and $\Pi'=\{s_0,\ldots,s_{2^{m'}}\}$ be any partitions of the intervals $[a,b]$ and $[c,d]$, respectively. Then, for any function
 $F:[a,b]\times[c,d]\to\mathbb{R}$ with finite $(p,q)$-bivariation, the following inequality holds
 $$\sum_{i=1}^{2^m}\sum_{j=1}^{2^{m'}} |\Delta_i\Delta_j F(t_i,s_j)| \leq 4\cdot 2^{m+m'}\rho\left( \frac{\|F\|_{1;p}}{2^{m/p}}\right)\sigma\left(\frac{\|F\|_{2;q}}{2^{m'/q}}\right),$$
where $\rho$ and $\sigma$ are non-decreasing functions such that $\rho(u)\sigma(u)=u$.
 \end{lemma}
 \begin{proof}
  \begin{eqnarray*}
   \sum_{i=1}^m\sum_{j=1}^{m'} |\Delta_i\Delta_j F(t_i,s_j)| &\leq&\sum_{i=1}^m\sum_{j=1}^{m'} \Big[ |\Delta_j F(t_i,s_j)|+|\Delta_jF(t_{i-1},s_j)|\Big]\\
   &=& 2^{m'}\sum_{i=1}^m\left[\left(\sum_{j=1}^{2^{m'}}\frac{1}{2^{m'}}\Big[ |\Delta_j F(t_i,s_j)|+|\Delta_jF(t_{i-1},s_j)|\Big]\right)^{q}\right]^{1/q}\\
   &\leq& 2^{m'}\sum_{i=1}^m\left[\sum_{j=1}^{2^{m'}}\frac{1}{2^{m'}}\Big[ |\Delta_j F(t_i,s_j)|+|\Delta_jF(t_{i-1},s_j)|\Big]^{q}\right]^{1/q}\\
   &\leq& 2\cdot 2^{m'} \sum_{i=1}^{2^m}\left[\frac{1}{2^{m'}}\sum_{j=1}^{2^{m'}} \Big[ |\Delta_j F(t_i,s_j)|^q+|\Delta_jF(t_{i-1},s_j)|^q\Big]\right]^{1/q}\\
   &\leq& 4\cdot 2^{m+m'} \frac{\|F\|_{2;q}}{2^{m'/q}}.
  \end{eqnarray*}
A similar reasoning yields
$$\sum_{i=1}^m\sum_{j=1}^{m'} |\Delta_i\Delta_j F(t_i,s_j)|\leq 4\cdot 2^{m+m'} \frac{\|F\|_{1;p}}{2^{m/p}}.$$

From lemma \ref{boundsigma}, the result follows.
\end{proof}

\begin{proposition}\label{propdesig}
Let $F,G:[a,b]\times [c,d]\to\mathbb{R}$ be two functions, together with sequences of partitions $\Pi_0,\ldots,\Pi_M$, $\Pi_0',\ldots,\Pi_{M'}'$
of the intervals $[a,b]$ and $[c,d]$, respectively. Assume that assumptions of Theorem~\ref{tyoung} hold. Then, if we denote
$$S_{n,n'} = \sum_{i=1}^{2^n}\sum_{j=1}^{2^{n'}} F(t_i^{(n)},s_j^{(n')}) \Delta_i\Delta_j G(t_i^{(n)},s_j^{(n')}),$$
we have that
$$|S_{M,M'} - S_{0,M'}-S_{M,0}+S_{0,0}| \leq K \left(\sum_{m=1}^\infty \rho\left(\frac{\|F\|_{1;p}}{m^{1/p}}\right)\lambda\left(\frac{4}{m}\right)\right)\left(\sum_{m'=1}^\infty\sigma\left(\frac{\|F\|_{2;q}}{{m'}^{1/q}}\right)\mu\left(\frac{4}{m'}\right)\right),$$
where $K$ is an absolute constant.
 \end{proposition}

\begin{proof}
We begin by noting from Lemma \ref{d1d2S} that
$$\Delta_1\Delta_2 S_{k,k'} = -\sum_{i=1}^{2^k-1}\sum_{j=1}^{2^{k'}-1} \Delta_i\Delta_j F(t_{2i}^{(k)},s_{2j}^{(k')})\Delta_i\Delta_j G(t_{2i-1}^{(k)},s_{2j-1}^{(k')}).$$
Therefore,
$$|\Delta_1\Delta_2 S_{k,k'}|\leq \sum_{i=1}^{2^k-1}\sum_{j=1}^{2^{k'}-1} |\Delta_i\Delta_j F(t_{2i}^{(k)},s_{2j}^{(k')})| \lambda(2^{-k+2})\mu(2^{-k'+2}),$$
and from Lemma \ref{boundd1d2f}, we have
$$|\Delta_1\Delta_2 S_{k,k'}|\leq 4\cdot 2^{k+k'}\rho\left(\frac{\|F\|_{1;p}}{2^{k/p}}\right)\sigma\left(\frac{\|F\|_{2;q}}{2^{k'/q}}\right) \lambda(2^{-k+2})\mu(2^{-k'+2}).$$

Note that,
$$S_{M,M'} - S_{0,M'}-S_{M,0}+S_{0,0} = \sum_{k=1}^M\sum_{k'=1}^{M'} \Delta_1\Delta_2 S_{k,k'}.$$

Now, there is an elementary inequality (see, for instance, \cite[p. 181-182]{feng}) that says that if $f$ is non-decreasing and non-negative, the following bound holds true
$$\sum_{k=1}^\infty 2^{k-1}f\left(\frac{1}{2^k}\right) \leq \sum_{m=1}^\infty f\left(\frac{1}{m}\right).$$

Applying this inequality twice, we obtain
\begin{eqnarray*}
|S_{M,M'} - S_{0,M'}-S_{M,0}+S_{0,0}|&\leq&  \sum_{k=1}^n\sum_{k'=1}^{n'} |\Delta_1\Delta_2 S_{k,k'}|\\
&\leq&\sum_{k=1}^n\sum_{k'=1}^{n'} 4\cdot 2^{k+k'}\rho\left(\frac{\|F\|_{1;p}}{2^k/p}\right)\sigma\left(\frac{\|F\|_{2;q}}{2^{k'/q}}\right) \lambda(2^{-k+2})\mu(2^{-k'+2})\\
&\leq& 16 \sum_{m=1}^\infty \rho\left(\frac{\|F\|_{1;p}}{m^{1/p}}\right)\lambda\left(\frac{4}{m}\right)\sum_{m'=1}^\infty \sigma\left(\frac{\|F\|_{2;q}}{{m'}^{1/q}}\right)\mu\left(\frac{4}{m'}\right).
\end{eqnarray*}
This concludes the proof, and shows that $K\leq 16$.

\end{proof}

In a similar manner, but much more easily, one can prove the following lemma:
\begin{lemma}\label{desigmarginal}
Let $F,G:[a,b]\times [c,d]\to\mathbb{R}$ be two functions, together with sequences of partitions $\Pi_0,\ldots,\Pi_M$, $\Pi_0',\ldots,\Pi_{M'}'$ of the intervals $[a,b]$ and $[c,d]$, respectively. Assume that assumptions of Theorem~\ref{tyoung} hold. Then, if we denote
$$S_{M,M'} = \sum_{i=1}^{2^M}\sum_{j=1}^{2^{M'}} F(t_i^{(n)},s_j^{(n')}) \Delta_i\Delta_j G(t_i^{(n)},s_j^{(n')}),$$
we have that
$$|S_{M,0} - S_{0,0}| \leq K_1\mu(d-c) \sum_{m=1}^\infty\frac{\|F\|_{1;p}}{m^{1/p}}\lambda\left(\frac{4}{m}\right),$$
and
$$|S_{0,M'} - S_{0,0}| \leq K_2\lambda(b-a) \sum_{m'=1}^\infty\frac{\|F\|_{2;q}}{{m'}^{1/q}}\mu\left(\frac{4}{m'}\right),$$
where $K_1$ and $K_2$ are absolute constants.
\end{lemma}

\begin{remark}
One corollary of Proposition \ref{propdesig} is Theorem 4.1 in Young's original article \cite{young1}. Theorem 4.1 in \cite{young1} cannot be used directly to prove Theorem \ref{mainTh}, because it only works for integrands defined in terms of very specific double differences.
\end{remark}

Combining all the above results, we arrive at the following result.

\begin{proposition}\label{ineqstep}
Assume that $F,G:[a,b]\times[c,d]\rightarrow \mathbb{R}$ satisfy assumptions in Theorem~\ref{tyoung}. Then, for any step function $F_{\Pi(\xi),\Pi'(\zeta)}$, the following inequality holds

$$
\left|\int_a^b\int_c^d F_{\Pi(\xi),\Pi'(\zeta)}(x,y)d_{x,y}G(x,y)-F(b,d)(G(b,d)-G(a,d)-G(b,c)+G(a,c))\right|$$
$$\leq K \left(\sum_{m=1}^\infty \rho\left(\frac{\|F\|_{1;p}}{m^{1/p}}\right)\lambda\left(\frac{4}{m}\right)\right)\left(\sum_{m'=1}^\infty\sigma\left(\frac{\|F\|_{2;q}}{{m'}^{1/q}}\right)\mu\left(\frac{4}{m'}\right)\right)$$
\begin{equation}\label{prebound}
+K_1\mu(d-c) \sum_{m=1}^\infty\frac{\|F\|_{1;p}}{m^{1/p}}\lambda\left(\frac{4}{m}\right)+ K_2\lambda(b-a) \sum_{m'=1}^\infty\frac{\|F\|_{2;q}}{{m'}^{1/q}}\mu\left(\frac{4}{m'}\right).
\end{equation}


\end{proposition}
\begin{proof}
Let $F,G:[a,b]\times [c,d]\to\mathbb{R}$ be two functions, $\Pi(\xi), \Pi'(\eta)$ tagged partitions, together with sequences of partitions
$\Pi_0,\ldots,\Pi_M$, $\Pi_0',\ldots,\Pi_{M'}'$ of the intervals $[a,b]$ and $[c,d]$, respectively. Denoting
$$S_{n,n'} = \sum_{i=1}^{2^n}\sum_{j=1}^{2^{n'}} F_{\Pi(\xi),\Pi'(\eta)}(t_i^{(n)},s_j^{(n')}) \Delta_i\Delta_j G(t_i^{(n)},s_j^{(n')}),$$
we have that, for $M$ and $M'$,
$$\int_a^b \int_c^d F_{\Pi(\xi),\Pi'(\eta)}(x,y)d_{x,y} G(x,y) = S_{M,M'}.$$
Observe, also, that $S_{0,0} = F(b,d)(G(b,d)-G(a,d)-G(b,c)+G(a,c))$.
The result is thus a simple consequence of Proposition \ref{propdesig} and Lemma \ref{desigmarginal}.
\end{proof}

\begin{proof}[Proof of the main theorem]
From Proposition \ref{ineqstep}, the bound~(\ref{prebound}) holds uniformly for step functions of $F$. The 2D Young integral of $F$ w.r.t $G$ is defined as a Moore-Pollard-type limit of integrals of step functions and hence, we shall conclude the proof.
\end{proof}

\subsection{An Application to the Brownian Motion Local-Time}\label{applt}
In this section, we illustrate the importance of Theorem~\ref{mainTh} and Corollary~\ref{mainbound} with an application to local-times. In the sequel, $B = \{B(s); s\ge 0\}$ is a standard Brownian motion defined on some probability space $(\Omega,\mathcal{F},\mathbb{P})$. The goal of this section is o provide strong approximations for the two-parameter random integral process

\begin{equation}\label{st}
\int_{0}^t\int_{-2^m}^{2^m} g(s,x)d_{(s,x)}\ell^x(s); ~0\le t \le T,
\end{equation}
where
$$\ell^x(t): = \lim_{\varepsilon\rightarrow 0}\frac{1}{2\varepsilon}\int_0^t1\!\!1_{\{|B(s) - x| < \varepsilon\}} ds\quad \text{almost surely};~(t,x)\in [0,T]\times [-2^m,2^m]$$
is the so-called local-time of the Brownian motion on a bounded rectangle $[0,T]\times[-2^m,2^m]\subset \mathbb{R}^2$ with $m\in \mathbb{N}$, and $g:\Omega\times [0,T]\times [-2^m,2^m]\rightarrow \mathbb{R}$ is a two-parameter stochastic process with jointly continuous and controlled sample paths in the sense of~(\ref{s}). We are interested in strong approximations (in $L^1(\mathbb{P})$-sense) for (\ref{st}) in terms of the number of upcrossings of an embedded random walk based on $B$.

Similar to identity~(4.5) in~\cite{feng}, we shall write

\begin{eqnarray*}
\sum_{i=0}^{l-1}\sum_{j=0}^{p-1} g(s_j,x_i) \Delta_i\Delta_j \ell^{x_{i+1}}(s_j+1)&=&\sum_{i=1}^{l}\sum_{j=1}^{p}\ell^{x_{i}}(s_j)  \Delta_i\Delta_j g(s_j,x_i)\\
& &\\
&-& \sum_{i=1}^l\ell^{x_i}(t)\Delta_i g(t,x_i).
\end{eqnarray*}
From Lemmas 2.1,~2.2 in~\cite{feng}, we know that $\{\ell^x(s); 0\le s\le T; -2^m\le x \le 2^m\}$ has $(1,2+\delta)$-bivariations a.s for every $\delta > 0$. Then under conditions of Theorem~\ref{tyoung} and the classical 1D-Young integral~(see~\cite{young}), we have


\begin{equation}\label{int}
\int_0^t\int_{-2^m}^{2^m}g(s,x)d_{(s,x)}\ell^x(s) =\int_0^t\int_{-2^m}^{2^m}\ell(s,x)d_{(s,x)}g(s,x)  - \int_{-2^m}^{2^m}\ell^x(t)d_xg(t,x); ~0\le t \le T.
\end{equation}

\begin{remark}
The importance of Theorem~\ref{mainTh} (in particular, Corollary~\ref{mainbound}) lies on the fact that the paths of the Brownian local time is only known to be of finite $(1,2+\delta)$-bivariation (See Lemma 2.1 in~\cite{feng}) for any $\delta>0$. In this case, the usual Towghi inequality~(see Theorem~\ref{tow}) does not hold so the maximal inequality in Corollary~\ref{mainbound} plays a key role for the study of processes of the form~(\ref{int}).
\end{remark}
Since the 1D-Young integral in the right-hand side of~(\ref{int}) can be treated by means of standard Young estimates~(see~\cite{young}), we concentrate our example on the 2D-Young integral

\begin{equation}\label{st22}
\int_0^t\int_{-2^m}^{2^m}\ell(s,x)d_{(s,x)}g(s,x);~0\le t\le T.
\end{equation}
In the sequel, in order to approximate~(\ref{st22}), let us  introduce $T^k_0:= 0$ and

$$T^k_n: = \inf\{t> T^k_{n-1}; |B(t) - B(T^k_{n-1})|=2^{-k}\}; n\ge 1.$$
We set $A^k(t):= \sum_{n=1}^\infty B(T^k_n)1\!\!1_{\{T^k_n < t \le T^k_{n+1}\}}; 0\le t \le T$. In the sequel, for a given $x\in \mathbb{R}$, let $j_k(x)$ be the unique integer such that $(j_k(x)-1)2^{-k} < x \le j_k(x)2^{-k}$. Let us define

$$u(j_k(x)2^{-k},k,t):= \#\ \Big\{n \in \{0, \ldots, N^k(t)-1\}; A^k(T^k_{n}) =(j_k(x)-1)2^{-k}, A^k(T^k_{n+1}) =j_k(x)2^{-k}\Big\};$$
for~$x\in \mathbb{R}, k\ge 1, 0\le t \le T.$ Here, $N^k(t): = \max \{n; T^k_n \le t\}$ is the length of the embedded random walk until time $t$. By the very definition, $u(j_k(x)2^{-k},k,t) :=$ number of upcrossings of~$A^k$~from~$(j_k(x)-1)2^{-k}$~to~$j_k(x)2^{-k}$~before time~$t$. To shorten notation, we denote

$$U^k(t,x):= 22^{-k}u(j_k(x)2^{-k},k,t);x\in I_m, 0\le t\le T,$$
where $I_m:= [-2^m,2^m]$ for a given positive integer $m\ge 1$.

\

\noindent Assumption~\textbf{(H1)}:
Let $g^k:\Omega\times [0,T]\times I_m\rightarrow\mathbb{R}$ be a sequence of stochastic processes such that

$$g^k(t,y)\rightarrow g(t,y)~a.s~\text{uniformly in}~(t,y)\in [0,T]\times I_m$$
and $g$ has jointly continuous paths a.s.

\

\noindent Assumption~\textbf{(H2.1)}: Assume for every $L>0$, there exists a positive constant $M$ such that

\begin{equation}\label{l2.1}
|\Delta_i\Delta_jg(t_i,x_j)|\le M|t_i-t_{i-1}|^{\frac{1}{q_1}}|x_j-x_{j-1}|^{\frac{1}{q_2}}~a.s
\end{equation}
for every partition $\Pi= \{t_i\}_{i=0}^N\times \{x_j\}_{j=0}^{N^{'}}$ of $[0,T]\times [-L,L]$, where $q_1,q_2>1$. In addition, there exists $\alpha\in (0,1)$ and $\delta >0$ such that $\min\{\alpha + \frac{1}{q_1}, \frac{1-\alpha}{2+\delta} + \frac{1}{q_2}\} > 1$.



\

\noindent Assumption~(\textbf{H2.2}): In addition to assumption \textbf{(H2.1)}, let us assume  $\forall L >0$, there exists $M>0$ such that

\begin{equation}\label{l2.3}
\sup_{k\ge 1} |\Delta_i\Delta_jg^k(t_i,x_j)|\le M |t_i-t_{i-1}|^{\frac{1}{q_1}}|x_j-x_{j-1}|^{\frac{1}{q_2}} ~a.s.
\end{equation}
for every partition $\Pi= \{t_i\}_{i=0}^N\times \{x_j\}_{j=0}^{N^{'}}$ of $[0,T]\times [-L,L]$.


\begin{remark}
Concrete examples for $(g^k,g)$ in terms of suitable functional derivatives of a given non-anticipative functional $F_t:C([0,t];\mathbb{R})\rightarrow \mathbb{R}$ of Brownian paths are illustrated by~\cite{LOS} in the framework of functional It\^o formulas. In particular, the authors show that suitable 2D-Young integral w.r.t local-times represents the unbounded variation components for functionals $F_t:C([0,t];\mathbb{R})\rightarrow \mathbb{R}$ of the Brownian paths under controlled sample paths assumptions. We refer the reader to Section 8.1-8.2 in~\cite{LOS} for further details.
\end{remark}
Now let us recall a technical lemma describing some necessary bounds for the number of upcrossings. In the sequel, we always consider the stopped Brownian motion at $S_m := \inf\{t\ge 1 ; |B(t)|> 2^m\}\wedge T$.

\begin{lemma}\label{lemmaLk}
For each $m\ge 1$, the following properties hold:

\

\noindent (i) $U^{k}(t,x)\rightarrow \ell^x(t)\quad\text{a.s uniformly in}~(x,t)\in I_m\times [0,T]$ as $k\rightarrow \infty.$

\

\noindent (ii) $\sup_{k\ge 1}\mathbb{E}\sup_{x\in I_m}\|U^{k}(\cdot,x)\|^p_{[0,T]; 1}< \infty$ and $\sup_{k\ge 1}\sup_{x\in I_m}\|U^{k}(\cdot,x)\|_{[0,T]; 1}< \infty~a.s$ for every $p\ge 1$.


\

\noindent (iii) $\sup_{k\ge 1}\mathbb{E}\sup_{t\in [0,T]}\| U^k(t)\|^{2+\delta}_{I_m;2+\delta}< \infty$ and $\sup_{k\ge 1}\sup_{t\in [0,T]}\| U^k(t)\|^{2+\delta}_{I_m;2+\delta}< \infty~a.s$~for every $\delta>0$.

\begin{proof}
The proof can be founded in Corollary 2.1 in~\cite{OS} and Lemma 8.1 in~\cite{LOS}.
\end{proof}


\end{lemma}

\begin{proposition}
Under assumption \textbf{(H1-H2)}, the following approximation holds

\begin{equation}\label{convergence}
\int_0^t\int_{-2^m}^{2^m} U^{k}(s,x)dg^k(s,x)\rightarrow \int_0^t\int_{-2^m}^{2^m}\ell(s,x)dg(s,x)\quad \text{in}~L^1(\mathbb{P})
\end{equation}
as $k\rightarrow \infty$, for every $t\in [0,T]$.
\end{proposition}
\begin{proof}
By \textbf{(H2.2)}, (ii,iii) in Lemma~\ref{lemmaLk} and Theorem~\ref{tyoung}, we know that the 2D Young integral $\int_0^t\int_{-2^m}^{2^m} U^{k}(s,x)dg^k(s,x)$ exists for every $k\ge 1$. We apply Lemma~\ref{lemmaLk}, Th. 6.3 and 6.4 in~\cite{young1} to get

$$\lim_{k\rightarrow \infty}\int_0^t\int_{-2^m}^{2^m} U^{k}(s,x)dg^k(s,x) =  \int_0^t\int_{-2^m}^{2^m}\ell(s,x)dg(s,x)$$
almost surely up to some vanishing conditions on $t=0$ and $x=-2^m$. They clearly vanish for $t=0$. For $x=-2^m$ we have to work a little. In fact, we will enlarge our domain in $x$, from $[-2^m,2^m]$ to $[-2^m-1,2^m]$, and define for all functions with $-2^m-1\le x < -2^m$, the value $0$, that is, for the functions $U^{k}(s,x), \ell^x(s), g(s,x)$ and $g^k(s,x)$ we put the value $0$, whenever $x\in [-2^m-1,-2^m)$. Then, it is easy to see, that all the conclusions of Lemma \ref{lemmaLk} still hold true, and in this case, $U^{k}(s,-2^m-1)=0$ and $\ell^{-2^m-1}(s)=0$ for all $s$. Thus, we can apply Theorems~6.3 ad 6.4 in~\cite{young1} on the interval $[0,t]\times [-2^m-1,2^m]$. It remains to show uniform integrability. By Corollary~\ref{mainbound}, we have

\begin{eqnarray}
\nonumber \Bigg|\int_0^t\int_{-2^m}^{2^m}U^k(s,x)d_{(s,x)}g^k(s,x)\Bigg|&\le& K_0U^k(2^m,T) + K\|U^k\|^{\alpha}_{1;1} \|U^k\|^{1-\alpha}_{2;2+\delta}\\
\nonumber & &\\
\label{better}&+& K_1\|U^k\|_{1;1} + K_2\| U^k  \|_{2;2+\delta}.
\end{eqnarray}
Here $K_0$ is a constant which comes from assumption~(\ref{l2.3}) and $K,K_1,K_2$ are positive constants which only depend on the constants of assumption $\textbf{(H2.1, H2.2)}$ namely $\alpha,q_1,q_2,\delta, T, m$. From Lemma~\ref{lemmaLk}, we have $\sup_{k\ge 1}\mathbb{E}\|U^{k}\|^{2+\delta}_{2;2+\delta} < \infty$, $\sup_{k\ge 1}\mathbb{E}\|U^{k}\|^{r}_{1;1} < \infty$ for every $r\ge 1$ and $\delta>0$. From Th.1 in~\cite{Barlow}, $\{U^{k}(2^m,T);k\ge 1\}$ is uniformly integrable, so we only need to check uniform integrability of $\{\|U^{k}\|^\alpha_{1;1}\|U^{k}\|^{1-\alpha}_{2;2+\delta}; k\ge 1 \}$. For $\beta > 1$, we apply H\"{o}lder inequality to get

$$\mathbb{E}\|U^{k}\|^{\beta\alpha}_{1;1}\|U^{k}\|^{(1-\alpha)\beta}_{2;2+\delta}\le \big(\mathbb{E}\|U^k\|_{2;2+\delta}\big)^{1/b}\big(\mathbb{E}\|U^k\|^{\alpha\beta d}_{1;1}\big)^{1/d};~k\ge 1$$
where $b=\frac{1}{(1-\alpha)\beta} >1,~d = \frac{b}{b-1} = \frac{1}{1-(1-\alpha)\beta}$ with $\alpha \in (0,1)$. Lemma~\ref{lemmaLk} allows us to conclude the proof.



\end{proof}


\end{document}